\documentclass{article}

\usepackage{amsfonts}
\usepackage{amsthm}
\usepackage{amsmath}
\usepackage{amssymb}
\usepackage{mathtools} 
\usepackage{authblk} 
\usepackage{xcolor}
\usepackage{soul} 
\usepackage{mathrsfs}

\usepackage[top=1in, bottom=1in, left=1in, right=1in]{geometry}
\newcommand{\norm}[1]{\left \lVert#1\right \rVert}
\newtheorem{thm}{Theorem}[section]
\newtheorem{lem}{Lemma}[section]

\newtheorem{assumption}{Assumption}[section]
\theoremstyle{definition}
\newtheorem{defn}{Definition}[section]

\theoremstyle{remark}
\newtheorem{remark}{Remark}[section]

\definecolor{cucol}{rgb}{0,0,0.8}
\definecolor{afcol}{rgb}{1,0,0}

\numberwithin{equation}{section}

\begin{document}
	
		
		\title{Asymptotic separation of solutions to fractional stochastic multi-term differential equations}
		
		\date{}
		


		\author[]{Arzu Ahmadova \thanks{Email: \texttt{arzu.ahmadova@emu.edu.tr}}}
		\author[]{Nazim I. Mahmudov\thanks{ Corresponding author. Email: \texttt{nazim.mahmudov@emu.edu.tr}}}

		\affil{Department of Mathematics, Eastern Mediterranean University, Gazimagusa, 99628, TRNC, Mersin 10, Turkey}
		
		
		\maketitle
		

		\begin{abstract}
			\noindent In this paper, we study the exact asymptotic separation rate of two distinct solutions of Caputo stochastic multi-term differential equations (Caputo SMTDEs for short). Our goal in this paper is to establish results on the global existence and uniqueness and continuity dependence of the initial values of the solutions to Caputo SMTDEs with non-permutable matrices of order $\alpha \in (\frac{1}{2},1)$ and $\beta \in (0,1)$ whose coefficients satisfy a standard Lipschitz condition. For this class of systems, we then show the asymptotic separation property between two different solutions of Caputo SMTDEs with a more general condition based on $\lambda$. Also, the asymptotic separation rate for the two distinct mild solutions reveals that our asymptotic results are general.

	\end{abstract}
	
		
		
\section{Introduction}\label{Sec:intro}

Over the years, many results on the theory and applications of stochastic differential equations (SDEs) have been studied \cite{oksendal, ito,prato}. In particular, fractional stochastic differential equations (FSDEs), which are a generalisation of differential equations by using fractional and stochastic calculus are becoming more popular due to their applications in modeling and financial mathematics. The nonlinear system of FDEs has been studied from various points of view: applications to population dynamics, optimal pricing in economics, and recent COVID-19 epidemics. Recently, FSDEs have been extensively used for modeling mathematical problems in finance \cite{tien,farhadi}, dynamics of complex systems in engineering \cite{wang} and other areas \cite{gangaram,mandelbrot}. Most results on fractional stochastic dynamical systems are limited to proving the existence and uniqueness of mild solutions using the fixed point theorem \cite{arzu}-\cite{shen}. A study on different types of stability studies for FSDEs can be found in \cite{luo2,mao2,wangj}.

Using  fractional  derivatives  instead  of  integer-order  derivatives  allows  for  the  modelling  of  a  wider  variety  of  behaviours. But  sometimes,  SDEs  involving  one  fractional  order  of  differentiation  are  not  sufficient  to  describe  physical  processes.  Therefore,  recently,  several  authors  have  studied  more  general  types  of  fractional-order stochastic models,  such  as  multi-term  equations to get analytical and numerical approximation results. For instance, the authors in \cite{arzu-sle} have studied Euler-Maruyama scheme for fractional stochastic Langevin multi-term equations and introduced a general form of FSLEs together with strong convergence rate of a numerical mild solution.

Among the many scientific articles on asymptotic behaviour and asymptotic separation of fractional stochastic differential equations, we will mention only a few that motivate this work:
\begin{itemize}
	\item A few works on asymptotic separation of two distinct solutions to fractional stochastic differential equations which can also be found in \cite{doan}. Similar work on an exact asymptotic separation rate of two distinct solutions of doubly singular stochastic Volterra integral equations (SVIEs) with two different initial values was discussed in \cite{li-hu}.
	\item Results on the asymptotic behavior of solutions of fractional differential equations with fractional time derivative of Caputo type are relatively rare in the literature. In \cite{cong}, Cong et al. study the asymptotic behavior of solutions of the perturbed linear fractional differential system. Cong et al. \cite{stefan} proved the theorem of linearized asymptotic stability for fractional differential equations. 
	\item The authors in \cite{arzu_asymptoticstability} studied the existence and asymptotic stability at the $p$-th moment of a mild solution for a class of nonlinear fractional neutral stochastic differential equations. The results are obtained with the help of the theory of fractional differential equations, some properties of Mittag-Leffler functions and its asymptotic analysis under the assumption that the corresponding fractional stochastic neutral dynamical system is asymptotically stable. The similar asymptotic stability result at the $p$-th moment of a mild solution of nonlinear impulsive stochastic differential equations was discussed in \cite{mahmudov-sakthivel,luo1}.

\end{itemize}

While there are several papers on deterministic fractional multi-term differential equations (FMTDEs) with constant and variable coefficients (see \cite{arzu-ismail,arran-joel}), there are few papers dealing with stochastic Caputo stochastic multi-term differential equations (Caputo SMTDEs) involving non-permutable matrices. The lack of asymptotic separation of the solutions of this class of Caputo SMTDEs motivates us to develop new results on asymptotic analysis. The main contributions of our work are described in detail below:
\begin{enumerate}
	\item To the best of our knowledge, we study asymptotic separation between two distinct mild solutions rather than integral equations. This is a lucky consequence which forms an interesting result in its own right.
	\item We consider more general Caputo-FSDEs with non-permutable matrices under the weaker condition $\lambda > \alpha$, which is true even in the special case when $\beta=0$ and $A, B$ are equal to zero matrices, than the condition represented in \cite{doan}. With respect to this condition, the asymptotic distance between solutions is greater than $t^{-\alpha-\epsilon}$ as $t \to \infty$ for any $\epsilon > 0$.
	\item  We obtain a bound for the leading coefficient of the asymptotic separation rate for the two distinct solutions which reveals that our asymptotic results are general. 
\end{enumerate}

Hence, the plan of this paper is systematized as follows: Section \ref{Sec:prel} is an introductory section in which we recall the main definitions, results from fractional calculus, and necessary lemmas from fractional differential equations, and in Section \ref{sec:3} we review to the framework for the main results of the theory. Section \ref{stochastic} is devoted to proving global existence and uniqueness and continuity dependence on the initial values of the solutions of Caputo SMTDEs of order $\alpha \in (\frac{1}{2},1)$ and $\beta \in (0,1)$ with non-permutable matrices. In Section \ref{asymptotic}, we investigate new results on the asymptotic behavior of solutions of the Caputo SMTDE by studying an asymptotic separation between two different solutions. In Section \ref{sec:example} we present an example to verify the results proved in Section \ref{stochastic} and \ref{asymptotic}. Section \ref{concl} is for the conclusion.

\section{Mathematical preliminaries}\label{Sec:prel}

We  embark  on  this  section  by  briefly  introducing  the  essential  structure  of  fractional  calculus  and  fractional  differential operators.  For  the  more  salient details  on  these  matters,  see  the  textbooks \cite{podlubny},\cite{miller-ross}-\cite{oldham-spanier}.  Note  that  none  of  the  results  in  this  section  are  new,  except Definition 2.3 and 2.4 (they are recently defined in \cite{nai} and \cite{arran-ozarslan}, respectively)
at the end, which we place in this section since a representation of solution is the Mittag-Leffler type function and it will be used later in the paper. 

\begin{defn}[\cite{podlubny,kilbas}]
	The Riemann--Liouville integral operator of fractional order $\alpha>0$ is defined by
	\begin{equation}
	\label{RLint:def}
	I^{\alpha}_{0+}g(t)=\frac{1}{\Gamma(\alpha)}\int_0^t(t-r)^{\alpha-1}g(r)\,\mathrm{d}r \\,\quad \text{for} \quad t>0.
	\end{equation}
	The Riemann--Liouville derivative operator of fractional order $\alpha>0$ is defined by
	\begin{equation}
	D^{\alpha}_{0+}g(t)=\frac{\mathrm{d}^n}{\mathrm{d}t^n}\left(I^{n-\alpha}_{0+}g(t)\right), \quad \text{where}\quad n-1<\alpha \leq n.
	\end{equation}
\end{defn}

\begin{defn}[\cite{podlubny}]
	Suppose that $\alpha>0$, $t>0$. The Caputo derivative operator of fractional order $\alpha$ is defined by:
	\begin{equation*}
	\prescript{C}{}D^{\alpha}_{0+}g(t)=\prescript{}{}I^{n-\alpha}_{0+}\left(\frac{\mathrm{d}^n}{\mathrm{d}t^n}g(t)\right), \quad \text{where}\quad n-1<\alpha \leq  n.
	\end{equation*}
	In particular, for $\alpha \in (0,1)$
	\begin{equation*} 
	\prescript{}{}I^{\alpha}_{0^{+}}\prescript{C}{}D^{\alpha}_{0^{+}}f(t)=f(t)-f(0).
	\end{equation*}
\end{defn}
The Riemann--Liouville fractional integral operator and the Caputo fractional derivative have the following property for $\alpha \geq 0$ \cite{podlubny,kilbas}:
\begin{equation} \label{ID} \prescript{}{}I^{\alpha}_{0+}(\prescript{C}{}D^{\alpha}_{0+}g(t))=g(t)-\sum_{k=0}^{n-1}\frac{t^{k}g^{(k)}(0)}{\Gamma(k+1)}.
\end{equation}
The relationship between the Riemann--Liouville and Caputo fractional derivatives are as follows:
\begin{equation} \label{eq:relation}
\prescript{C}{}D^{\alpha}_{0+}g(t)=\prescript{}{}D^{\alpha}_{0+}g(t)-\sum_{k=0}^{n-1}\frac{t^{k-\alpha}g^{(k)}(0)}{\Gamma(k-\alpha+1)}, \quad \alpha\geq 0.
\end{equation}

\begin{defn}[\cite{nai}]
	We define a new Mittag-Leffler type function $ \mathscr{E}_{\alpha,\beta,\gamma}^{A,B}(\cdot) :\mathbb{R}\to \mathbb{R}$ generated by nonpermutable matrices $A,B\in\mathbb{R}^{n\times n}$ as follows:
	\begin{equation}
	\label{ML4}
	\mathscr{E}_{\alpha,\beta,\delta}^{A,B}(t)\coloneqq\sum_{k=0}^{\infty}\sum_{m=0}^{\infty}\mathcal{Q}_{k,m}^{A,B}\frac{t^{k\alpha+m\beta}}{\Gamma(k\alpha+m\beta+\delta)}, \quad \alpha,\beta>0, \quad \delta \in \mathbb{R},
	\end{equation}
	where $\mathcal{Q}_{k,m}^{A,B}\in \mathbb{R}^{n\times n} $, $k,m\in\mathbb{N}_{0}:\coloneqq \mathbb{N}\cup \left\lbrace 0 \right\rbrace $ is given by
	
	\begin{equation}\label{important}
	\mathcal{Q}_{k,m}^{A,B}\coloneqq\sum_{l=0}^{k} A^{k-l}B\mathcal{Q}_{l,m-1}^{A,B}, \quad  k,m\in\mathbb{N}, \qquad \mathcal{Q}_{k,0}^{A,B}\coloneqq A^{k}, \quad k\in\mathbb{N}_{0}, \qquad \mathcal{Q}_{0,m}^{A,B}\coloneqq B^{m}, \quad m\in\mathbb{N}_{0}.
	\end{equation}
\end{defn}
An explicit representation of  $\mathcal{Q}_{k,m}^{A,B}$ can be found in Table 1 in \cite{nai}. In the case of permutable matrices, i.e. $AB = BA$, we have
$\mathcal{Q}_{k,m}^{A,B}:=\binom{k+m}{m}A^{k}B^{m}$, $k,m\in\mathbb{N}_{0}$.
	
\begin{defn}[\cite{arran-ozarslan}]\label{Def:bML}
	We consider the Mittag-Leffler type function involving permutable matrices
	\begin{equation}
	\label{bML4}
	t^{\delta-1}E_{\alpha,\beta,\delta}(B t^{\alpha},A t^{\beta})=\sum_{k=0}^{\infty}\sum_{m=0}^{\infty}\binom{k+m}{m}\frac{A^k B^m}{\Gamma(k\alpha+m\beta+\delta)} t^{m\alpha+k\beta+\delta-1}.
	\end{equation}
\end{defn}

The following results are often used to compute estimations in Section \ref{stochastic} and \ref{asymptotic}.
\begin{lem} \label{lemtech}
	For all $\omega, t>0$, and $\alpha \in (\frac{1}{2}, 1)$ the following inequality holds:
	\begin{equation*}
	\frac{\omega}{\Gamma(2\alpha-1)}\int_{0}^{t}(t-r)^{2\alpha-2}E_{2\alpha-1}(\omega r^{2\alpha-1})\mathrm{d}r \leq E_{2\alpha-1}(\omega t^{2\alpha-1}).
	\end{equation*}
\end{lem}
\begin{proof}
	Applying the series representation of Mittag-Leffler function and definition of beta function, then by swapping summation and integration, we get the desired result.
\end{proof} 
The following lemma plays a necessary role on proofs of main results in Section \ref{stochastic} and \ref{asymptotic}.

\begin{lem}[H\"{o}lder's inequality]
	Suppose that $\alpha_1$,$\alpha_2 >1$ and $\frac{1}{\alpha_1}+\frac{1}{\alpha_2}=1$. If $|f(t)|^{\alpha_1}$, $|g(t)|^{\alpha_2}\in L^{1}(\Omega)$, then $|f(t)g(t)|\in L^{1}(\Omega)$ and 
	\begin{equation*}
	\int_{\Omega}|f(t)g(t)|\mathrm{d}t \leq \left( \int_{\Omega}|f(t)|^{\alpha_1}\mathrm{d}t\right)^{\frac{1}{\alpha_1}} \left(\int_{\Omega}|g(t)|^{\alpha_2}\mathrm{d}t\right)^{\frac{1}{\alpha_2}},
	\end{equation*}
	where $L^{1}(\Omega)$ represents the Banach space of all Lebesgue measurable functions $f:\Omega \to \mathbb{R}$ with \\
	$\int_{\Omega}|f(t)|\mathrm{d}t < \infty$. Especially, when $\alpha_1=\alpha_2=2$, the H\"{o}lder's inequality  reduces to the Cauchy-Schwartz  inequality 
	\begin{equation}\label{cauchy}
	\left( \int_{\Omega}|f(t)g(t)|\mathrm{d}t \right) ^{2}\leq \int_{\Omega}|f(t)|^{2}\mathrm{d}t  \int_{\Omega}|g(t)|^{2}\mathrm{d}t.
	\end{equation} 
\end{lem}

\begin{lem}[Jensen's inequality]
	Let $n \in \mathbb{N}$, $q>1$ and $x_{i}\in\mathbb{R}_{+}$, $i=1,2,\ldots,n$. Then, the following inequality holds true:
	\begin{equation*}
	\|\sum_{i=1}^{n}x_{i}\|^{q} \leq n^{q-1}\sum_{i=1}^{n}\|x_{i}\|^{q}.
	\end{equation*}
	In particular, we consider the following inequality with $q=2$ within the estimations on this paper:
	\begin{equation}\label{ineqq}
	\|\sum_{i=1}^{n}x_{i}\|^{2} \leq n\sum_{i=1}^{n}\|x_{i}\|^{2}.
	\end{equation}
\end{lem}

\section{Formulation of main problem} \label{sec:3}
In this section, we resort main assumptions which will be used in throughout the Section \ref{stochastic}. We consider a Caputo stochastic multi-term differential equation of order $\alpha\in (\frac{1}{2},1)$ and $\beta \in (0,1)$ with non-permutable matrices has the following form
\begin{equation}\label{fstoc}
\begin{cases}
\left( \prescript{C}{}D_{0^{+}}^{\alpha}X \right) (t)-A \left( \prescript{C}{}D_{0^{+}}^{\beta} X \right) (t)-B X(t)=b(t,X(t))+\sigma(t, X(t))\frac{\mathrm{d}W(t)}{\mathrm{d}t}, \\
X(0)=\eta,
\end{cases}
\end{equation}
The coefficients $b,\sigma : [0,T] \times \mathbb{R}^{n} \to \mathbb{R}^{n}$ are measurable and bounded functions. $A,B \in \mathbb{R}^{n\times n}$ are non-permutable matrices. 
We introduce the norm of the matrix which are used throughout this paper. For any matrix $A=(\,a_{ij})\,_{n\times n} \in \mathbb{R}^{n\times n}$, the norm of the matrix $A$, according to the maximum norm on $\mathbb{R}^{n}$ is $\|A\|=\max_{1\leq i\leq  n}\sum_{j=1}^{n}|a_{ij}|$.
Moreover, let $( W_{t})_{t\geq 0} $ denote a standard scalar Brownian motion on a complete probability space $(\Omega,\mathscr{F}, \mathbb{F}, \mathbb{P})$ with filtration $\mathbb{F} \coloneqq \left\lbrace \mathscr{F}_{t}\right\rbrace_{t \in [0,T]}$. The initial condition $\eta$ is an $\mathscr{F}_{0}$-measurable $H$-random variable. For each $t \in [0,T]$, let $\Xi_{T} \coloneqq \mathbb{L}^{2}(\Omega, \mathscr{F}_{T},\mathbb{P})$ denote the space of all $\mathscr{F}_{T}$ measurable, mean square integrable functions $f=(f_{1}, f_{2}, \ldots, f_{n})^{t}: \Omega \to \mathbb{R}^{n}$ with
\begin{equation*}
\|f\|^{2}_{\text{ms}}\coloneqq \textbf{E}(\|f\|^{2}).
\end{equation*}
Let $H^{2}([0,T], \mathbb{R}^{n})$ be well-endowed with the weighted maximum norm $\|\cdot\|_{\omega}$ as
\begin{equation} \label{maxnorm}
\| \xi\|^{2}_{\omega} \coloneqq \sup_{t\in [0,T]}\frac{\textbf{E}\| \xi (t)\|^{2}}{E_{2\alpha-1}(\omega t^{2\alpha-1})}, \\  \quad \text{for all} \quad \xi \in H^{2}([0,T], \mathbb{R}^{n}),  \omega>0.
\end{equation}

Let $\mathbb{R}^{n}$ be endowed with the standard Euclidean norm and $H^{2}([0,T], \mathbb{R}^{n})$ denote the space of all $\mathbb{F}_{T}$-measurable processes $\xi$ satisfying 

\begin{align*}
\| \xi\|^{2}_{H^{2}} \coloneqq \sup_{t\in [0,T]} \textbf{E}\|\xi(t)\|^{2}< \infty.
\end{align*}

Obviously, $(H^{2}([0,T], \mathbb{R}^{n}), \| \cdot \|_{H^{2}})$ is a Banach space. Since two norms $ \| \cdot \|_{H^{2}}\quad \text{and} \quad \| \cdot \|_{\omega}$ are equivalent,  $(H^{2}([0,T], \mathbb{R}^{n}),\| \cdot \|_{\omega})$ is also Banach space. Therefore, the set $H^{2}_{\eta}([0,T], \mathbb{R}^{n})$ is complete with respect to the norm $\|\cdot\|_{\omega}$ and $b, \sigma:[0,T]\times \mathbb{R}^{n} \to \mathbb{R}^{n}$ are measurable and bounded functions satisfying the following conditions:
\begin{assumption} \label{A1}
	The drift $b$ and diffusion $\sigma$ terms satisfy global Lipschitz continuity: there exists $L_{b},L_{\sigma}>0$ such that for all $x,y\in \mathbb{R}^{n}$ , $t\in [0,T]$,
	\begin{equation*}
	\|b(t,x)-b(t,y)\| \leq L_{b}\|x-y\|, \qquad  \|\sigma(t,x)-\sigma(t,y)\| \leq L_{\sigma}\|x-y\|.
	\end{equation*}
\end{assumption}
\begin{assumption}\label{A11}
	$b(\cdot,0)$ is $\mathbb{L}^{2}$ integrable i.e.
	\begin{equation*}
	\int_{0}^{T} \|b(r,0)\|^{2}\mathrm{d}r <\infty,
	\end{equation*}
	and $\sigma(\cdot,0)$ is essentially bounded i.e.
	\begin{equation*}
	ess\sup\limits_{r\in [0,T]}\|\sigma(r,0)\|< \infty.
	\end{equation*} 
\end{assumption}

\begin{defn}
	A stochastic process $ \left\lbrace X(t), t\in [0,T]\right\rbrace$ is called  a mild solution of \eqref{fstoc} if 
	\begin{itemize}
		\item $X(t)$ is adapted to $\left\lbrace \mathscr{F}_{t}\right\rbrace _{t \geq 0}$ with 
		$ \int_{0}^{t}  \|X(t)\|^{2}_{H^{2}}\mathrm{d}t< \infty$ almost everywhere;
		\item $X(t)\in H^{2}([0,T], \mathbb{R}^{n})$ has continuous path on $t\in [0,T]$ a.s. and satisfies Volterra integral equation of second kind on $t\in [0,T]$:
	\end{itemize}	
\end{defn}	
	\begin{align}  \label{integral equation}
	X(t)=\eta-\frac{A t^{\alpha-\beta}}{\Gamma(\alpha-\beta+1)}\eta&+\frac{A}{\Gamma(\alpha-\beta)}\int_{0}^{t}(t-r)^{\alpha-\beta-1}X(r) \mathrm{d}r\nonumber\\
	&+\frac{B}{\Gamma(\alpha)}\int_{0}^{t}(t-r)^{\alpha-1}X(r)\mathrm{d}r\nonumber\\
	&+\frac{1}{\Gamma(\alpha)}\int_{0}^{t}(t-r)^{\alpha-1}b(r,X(r))\mathrm{d}r\nonumber\\
	&+\frac{1}{\Gamma(\alpha)}\int_{0}^{t}(t-r)^{\alpha-1}\sigma(r,X(r))\mathrm{d}W(r).
	\end{align}
	
	To define above integral equation, we apply Riemann--Liouville integral operator $I^{\alpha}_{0+}$ to the both side of \eqref{fstoc}, we define
	\begin{equation*}
	I^{\alpha}_{0+}\prescript{C}{}D^{\alpha}_{0+} X(t)-A I^{\alpha}_{0+}\prescript{C}{}D^{\beta}_{0+} X(t)-B I^{\alpha}_{0+}X(t)=I^{\alpha}_{0+}b(t,X(t))+I^{\alpha}_{0+}\sigma(t, X(t))\frac{\mathrm{d}W(t)}{\mathrm{d}t}
	\end{equation*}
	Then we use the relationship between Riemann--Liouville integral and Caputo fractional differential operators \eqref{ID} for $\frac{1}{2}<\alpha\leq 1$, and $0<\beta\leq 1$, we get
	\begin{align*}
	X(t)&=\eta+\frac{A}{\Gamma(\alpha)}\int_{0}^{t}(t-r)^{\alpha-1}\prescript{C}{}D^{\beta}_{0+} X(r) \mathrm{d}r+\frac{B}{\Gamma(\alpha)}\int_{0}^{t}(t-r)^{\alpha-1}X(r)\mathrm{d}r\\
	&+\frac{1}{\Gamma(\alpha)}\int_{0}^{t}(t-r)^{\alpha-1}b(r,X(r))\mathrm{d}r+\frac{1}{\Gamma(\alpha)}\int_{0}^{t}(t-r)^{\alpha-1}\sigma(r,X(r))\mathrm{d}W(r),
	\end{align*}
	where 
	\begin{align} \label{integr}
	\frac{\lambda}{\Gamma(\alpha)}\int_{0}^{t}(t-r)^{\alpha-1}\prescript{C}{}D^{\beta}_{0+} X(r) \mathrm{d}r=\frac{\lambda}{\Gamma(\alpha-\beta+1)}\int_{0}^{t}(t-u)^{\alpha-\beta}X^{\prime}(u)\mathrm{d}u,
	\end{align}
Then, we apply integration by parts formula for \eqref{integr} to get \eqref{integral equation}.
Now we can represent our mild solution of \eqref{fstoc} involving non-permutable matrices.

\begin{lem}
	Let $A,B\in \mathbb{R}^{n\times n}$ with non-zero commutator, i.e., $\left[ A, B\right] \coloneqq AB- BA \neq 0$. For each initial value $\eta \in \Xi_{0}$, the mild solution $X(\cdot)\in \mathbb{R}^{n}$ of the Cauchy problem \eqref{fstoc} can be represented in terms of Mittag-Leffler type functions involving non-permutable matrices as below:
	\allowdisplaybreaks
	\begin{align}
	X(t)&=\left(I+ \sum_{k=0}^{\infty}\sum_{m=0}^{\infty}\mathcal{Q}_{k,m}^{A,B}B\frac{t^{k(\alpha-\beta)+m\alpha+\alpha}}{\Gamma(k(\alpha-\beta)+m\alpha+\alpha+1)}\right)\eta\nonumber\\
	&+\int\limits_{0}^{t}\sum_{k=0}^{\infty}\sum_{m=0}^{\infty}\mathcal{Q}_{k,m}^{A,B}\frac{(t-r)^{k(\alpha-\beta)+m\alpha+\alpha-1}}{\Gamma(k(\alpha-\beta)+m\alpha+\alpha)}X(r)\mathrm{d}r\nonumber\\
	&+\int\limits_{0}^{t}\sum_{k=0}^{\infty}\sum_{m=0}^{\infty}\mathcal{Q}_{k,m}^{A,B}\frac{(t-r)^{k(\alpha-\beta)+m\alpha+\alpha-1}}{\Gamma(k(\alpha-\beta)+m\alpha+\alpha)}b(r,X(r))\mathrm{d}r\nonumber\\
	&+\int\limits_{0}^{t}\sum_{k=0}^{\infty}\sum_{m=0}^{\infty}\mathcal{Q}_{k,m}^{A,B}\frac{(t-r)^{k(\alpha-\beta)+m\alpha+\alpha-1}}{\Gamma(k(\alpha-\beta)+m\alpha+\alpha)}\sigma(r,X(r))\mathrm{d}W(r)\nonumber\\
	&\coloneqq\left( I+t^{\alpha}\mathscr{E}_{\alpha-\beta,\alpha,\alpha+1}^{A,B}(t)B\right) \eta+\int\limits_{0}^{t}(t-r)^{\alpha-1}\mathscr{E}_{\alpha-\beta,\alpha,\alpha}^{A,B}(t-r)X(r)\mathrm{d}r\nonumber\\
	&+\int\limits_{0}^{t}(t-r)^{\alpha-1}\mathscr{E}_{\alpha-\beta,\alpha,\alpha}^{A,B}(t-r)b(r,X(r))\mathrm{d}r\nonumber\\
	&+\int\limits_{0}^{t}(t-r)^{\alpha-1}\mathscr{E}_{\alpha-\beta,\alpha,\alpha}^{A,B}(t-r)\sigma(r,X(r))\mathrm{d}W(r),\quad t>0.
	\end{align}
\end{lem}
	
\begin{lem}
As a special case, for each initial value $\eta \in \Xi_{0}$  the system \eqref{fstoc} has a unique mild solution in terms of Mittag-Leffler type functions \eqref{bML4} with permutable matrices on $[0,T]$ as below:
\begin{align}\label{36}
X(t)=\eta&+\eta t^{\alpha}BE_{\alpha,\alpha-\beta,\alpha+1}(B t^{\alpha}, At^{\alpha-\beta})\nonumber\\
&+\int_{0}^{t}(t-r)^{\alpha-1}E_{\alpha,\alpha-\beta,\alpha}(B (t-r)^{\alpha}, A (t-r)^{\alpha-\beta})X(r)\mathrm{d}r\nonumber\\
&+\int_{0}^{t}(t-r)^{\alpha-1}E_{\alpha,\alpha-\beta,\alpha}(B (t-r)^{\alpha}, A (t-r)^{\alpha-\beta})b(r,X(r))\mathrm{d}r\nonumber\\
&+\int_{0}^{t}(t-r)^{\alpha-1}E_{\alpha,\alpha-\beta,\alpha}(B (t-r)^{\alpha}, A (t-r)^{\alpha-\beta})\sigma(r,X(r))\mathrm{d}W(r).
\end{align}	
\end{lem}
These solutions can be derived with the help of variation of constants formula. Then the coincidence between the notion of mild solution and integral equation of \eqref{fstoc} with permutable and non-permutable matrices can be proved in a similar way depicted in \cite{arzu}. Therefore, we omit those proofs here.
\section{Existence \& uniqueness results and continuity dependence on initial conditions }\label{stochastic}
In Section \ref{asymptotic}, we will look at the behavior of solutions to multi-order systems as the independent variable goes to infinity. For this purpose, it is important to have an existence and uniqueness result. Therefore, our first aim in this research article is to show the global existence and uniqueness of solution of \eqref{fstoc}. Moreover, we also prove  the continuity dependence of solutions on the initial values.
\begin{thm}[Global existence and uniqueness and continuity dependence on the initial values of solutions of Caputo SMTDE]\label{thm2}
	Suppose that Assumptions \ref{A1} and \ref{A11} hold. Then
	
	(i) for any $\eta\in \Xi_{0}$, the Cauchy problem \eqref{fstoc} with initial condition $X(0)=\eta$ has a unique global solution on the whole interval $[0, T]$ denoted by $\varphi(\cdot, \eta)$;
	
	(ii) on any bounded time interval $[0,T]$ with $T>0$, the solution $\varphi(\cdot, \eta)$ depends continuously on $\eta$, i.e.
	\begin{equation*}
	\lim\limits_{\eta \to \gamma} \sup\limits_{t\in [0,T]}\|\varphi(t, \eta)-\varphi(t, \gamma)\|^{2}_{\text{ms}}=0.
	\end{equation*}
\end{thm}
For any $\eta\in \Xi_{0}$, we define an operator $\mathcal{T}_{\eta}: H^{2}_{\eta}([0,T],\mathbb{R}^{n}) \to H^{2}_{\eta}([0,T],\mathbb{R}^{n})$ by
\begin{align} \label{T.oper}
\mathcal{T}_{\eta} Y= \left( I+t^{\alpha}\mathscr{E}_{\alpha-\beta,\alpha,\alpha+1}^{A,B}(t)B\right) \eta&+\int\limits_{0}^{t}(t-r)^{\alpha-1}\mathscr{E}_{\alpha-\beta,\alpha,\alpha}^{A,B}(t-r)Y(r)\mathrm{d}r\nonumber\\
&+\int\limits_{0}^{t}(t-r)^{\alpha-1}\mathscr{E}_{\alpha-\beta,\alpha,\alpha}^{A,B}(t-r)b(r,Y(r))\mathrm{d}r\nonumber\\
&+\int\limits_{0}^{t}(t-r)^{\alpha-1}\mathscr{E}_{\alpha-\beta,\alpha,\alpha}^{A,B}(t-r)\sigma(r,Y(r))\mathrm{d}W(r),\quad t>0.
\end{align}
The following lemma is devoted to showing that $\mathcal{T}_{\eta}$ is well-defined. 
\begin{lem}\label{welldef}
	For $\eta \in \Xi_{0}$, the operator $\mathcal{T}_{\eta}$ is well-defined.
\end{lem}
\begin{proof}
	Let $Y\in H^{2}_{\eta}([0,T], \mathbb{R}^{n})$ be arbitrary. From the definition of $\mathcal{T}_{\eta}Y$ as in \eqref{T.oper} and the Jensen's inequality \eqref{ineqq} for $n=4$, we have for all $t \in [0,T]$:
	\allowdisplaybreaks
	\begin{align}\label{Tnorm}
	\|(\mathcal{T}_{\eta}Y)(t)\|^{2}_{\text{ms}} &\leq 4\|\eta\|^{2}_{\text{ms}}\| I+t^{\alpha}\mathscr{E}_{\alpha-\beta,\alpha,\alpha+1}^{A,B}(t)B\|^{2}\nonumber\\ &+4\textbf{E}\norm{\int\limits_{0}^{t}(t-r)^{\alpha-1}\mathscr{E}_{\alpha-\beta,\alpha,\alpha}^{A,B}(t-r)Y(r)\mathrm{d}r}^{2}\nonumber\\
	&+4\textbf{E}\norm{\int\limits_{0}^{t}(t-r)^{\alpha-1}\mathscr{E}_{\alpha-\beta,\alpha,\alpha}^{A,B}(t-r)b(r,Y(r))\mathrm{d}r}^{2}\nonumber\\
	&+4\textbf{E}\norm{\int\limits_{0}^{t}(t-r)^{\alpha-1}\mathscr{E}_{\alpha-\beta,\alpha,\alpha}^{A,B}(t-r)\sigma(r,Y(r))\mathrm{d}W(r)}^{2}.
	\end{align}
	Considering $\mathscr{M}\coloneqq \sup\limits_{t\in [0,T]}\|\mathscr{E}_{\alpha-\beta,\alpha,\alpha}^{A,B}(t)\|$ and using Cauchy-Schwarz inequality, we obtain the following results:
	\begin{align}\label{1}
    \textbf{E}\norm{\int\limits_{0}^{t}(t-r)^{\alpha-1}\mathscr{E}_{\alpha-\beta,\alpha,\alpha}^{A,B}(t-r)Y(r)\mathrm{d}r}^{2}&\leq t\int_{0}^{t}(t-r)^{2\alpha-2}\textbf{E}\|Y(r)\|^{2} \mathrm{d}r\nonumber\\
	&\leq \mathscr{M}^{2}T \int_{0}^{t}(t-r)^{2\alpha-2}\sup_{r\in [0,T]}\textbf{E}\|Y(r)\|^{2}\mathrm{d}r\nonumber\\
	&\leq  \mathscr{M}^{2}T \int_{0}^{t}(t-r)^{2\alpha-2}\|Y\|^{2}_{H^{2}}\mathrm{d}r=\mathscr{M}^{2}\frac{T^{2\alpha}}{2\alpha-1}\|Y\|^{2}_{H^{2}},
\end{align}
and
	\begin{align}\label{bb}
	\textbf{E}\norm{\int_{0}^{t}(t-r)^{\alpha-1}\mathscr{E}_{\alpha-\beta,\alpha,\alpha}^{A,B}(t-r)b(r,Y(r))\mathrm{d}r}^{2} \nonumber&\leq\mathscr{M}^{2}\int_{0}^{t}(t-r)^{2\alpha-2} \mathrm{d}r\textbf{E}\int_{0}^{t}\|b(r,Y(r))\|^{2}\mathrm{d}r\nonumber\\
	&=\mathscr{M}^{2}\frac{T^{2\alpha-1}}{2\alpha-1}\textbf{E}\int_{0}^{t}\|b(r,Y(r))\|^{2}\mathrm{d}r.
	\end{align}
	From Assumption \ref{A1}, we derive 
	\begin{align*}
	\|b(r,Y(r))\|^{2}&\leq 2\|b(r,Y(r))-b(r,0)\|^{2}+2\|b(r,0)\|^{2} \\
	&\leq 2L^{2}_{b}\|Y(r)\|^{2}+2\|b(r,0)\|^{2}.
	\end{align*}
	Therefore,
	\begin{align*}
	\textbf{E}\int_{0}^{t}\|b(r,Y(r))\|^{2}\mathrm{d}r&\leq 2L^{2}_{b}\textbf{E}\int_{0}^{t}\|Y(r)\|^{2}\mathrm{d}r+2\int_{0}^{t}\|b(r,0)\|^{2}\mathrm{d}r\\
	&\leq 2L_{b}^{2}T\sup\limits_{r \in [0,T]}\textbf{E}\|Y(r)\|^{2}+2\int_{0}^{T}\|b(r,0)\|^{2}\mathrm{d}r
	\end{align*}
	which together with \eqref{bb} implies that
	\begin{align}\label{bbb}
	\textbf{E}\norm{\int_{0}^{t}(t-r)^{\alpha-1}\mathscr{E}_{\alpha-\beta,\alpha,\alpha}^{A,B}(t-r)b(r,Y(r))\mathrm{d}r}^{2}&\leq 2\mathscr{M}^{2}\frac{L_{b}^{2}T^{2\alpha}}{2\alpha-1}\|Y\|^{2}_{H^{2}}\nonumber\\
	&+2\mathscr{M}^{2}\frac{T^{2\alpha-1}}{2\alpha-1}\int_{0}^{T}\|b(r,0)\|^{2}\mathrm{d}r.
	\end{align}
	Now using the It\^{o}'s isometry, we attain
	\begin{align*}
	\textbf{E}\norm{\int_{0}^{t}(t-r)^{\alpha-1}\mathscr{E}_{\alpha-\beta,\alpha,\alpha}^{A,B}(t-r)\sigma(r,Y(r))\mathrm{d}W(r)}^{2}
	&=\sum_{i=1}^{d}\textbf{E}\left( \int_{0}^{t}(t-r)^{\alpha-1}\mathscr{E}_{\alpha-\beta,\alpha,\alpha}^{A,B}(t-r)\sigma_{i}(r,Y(r))\mathrm{d}W_{r}\right)^{2} \\
	&=\mathscr{M}^{2}\sum_{i=1}^{d}\textbf{E}\left( \int_{0}^{t}(t-r)^{2\alpha-2}|\sigma_{i}(r,Y(r))|^{2}\mathrm{d}r\right) \\
	&=\mathscr{M}^{2}\textbf{E}\int_{0}^{t}(t-r)^{2\alpha-2}\|\sigma(r,Y(r))\|^{2}\mathrm{d}r\\
	&\leq \mathscr{M}^{2}T^{2\alpha-2}\textbf{E}\int_{0}^{t}\|\sigma(r,Y(r))\|^{2}\mathrm{d}r.
	\end{align*}
	From Assumption \ref{A1}, we also have,
	\begin{align*}
	\|\sigma(r,Y(r))\|^{2}&\leq 2\|\sigma(r,Y(r))-\sigma(r,0)\|^{2}+2\|\sigma(r,0)\|^{2} \\
	&\leq 2L^{2}_{\sigma}\|Y(r)\|^{2}+2\|\sigma(r,0)\|^{2}.
	\end{align*}
	Therefore, for all $t\in [0,T]$, we have
	\begin{align*}
	\textbf{E}\norm{\int_{0}^{t}(t-r)^{\alpha-1}\mathscr{E}_{\alpha-\beta,\alpha,\alpha}^{A,B}(t-r)\sigma(r,Y(r))\mathrm{d}W(r)}^{2}&\leq 2\mathscr{M}^{2}T^{2\alpha-2}L^{2}_{\sigma}\textbf{E}\int_{0}^{t}\|Y(r)\|^{2}\mathrm{d}r\\
	&+2\mathscr{M}^{2}T^{2\alpha-2}\int_{0}^{t}\|\sigma(r,0)\|^{2}\mathrm{d}r\\
	&\leq 2\mathscr{M}^{2}L_{\sigma}^{2}T^{2\alpha-1}\|Y(r)\|^{2}_{H^{2}}\\
	&+2\mathscr{M}^{2}T^{2\alpha-2}\int_{0}^{T}\|\sigma(r,0)\|^{2}\mathrm{d}r.
	\end{align*}
	This together with \eqref{Tnorm} and \eqref{1}-\eqref{bbb} yields that $\|\mathcal{T}_{\eta}Y\|^{2}_{H^{2}}< \infty$. Hence, the map $\mathcal{T}_{\eta}$ is well-defined.
\end{proof}
To prove global existence and uniqueness of solutions, we will show that the operator $\mathcal{T}_{\eta}$ is contractive with respect to the weighted maximum norm \eqref{maxnorm}. Now, we are in a position to prove Theorem \ref{thm2}.
 
\textbf{Proof of Theorem \ref{thm2}:} Let $T>0$ be an arbitrary. Choose and fix a positive constant $\omega$ such that 
\begin{equation}\label{omega}
\omega >4\Gamma(2\alpha-1)\mathscr{M}^{2}\Big(1+L_{b}^{2}T+L_{\sigma}^{2}\Big).
\end{equation}

(i) Choose and fix $\eta\in \Xi_{0}$. By virtue of Lemma \ref{welldef}, the operator $\mathcal{T}_{\eta}$ is well-defined. We will prove that the map $\mathcal{T}_{\eta}$ is a contraction with respect to the norm $\|\cdot\|_{\omega}$.

For this purpose, let $X, Y \in H^{2}([0,T],\mathbb{R}^{n})$ be arbitrary. From \eqref{T.oper} and the inequality \eqref{ineqq} with $n=3$, we derive the following estimations for all $t \in [0,T]$:
\begin{align*}
\textbf{E}\|(\mathcal{T}_{\eta}X)(t)-(\mathcal{T}_{\eta}Y)(t)\|^{2}&\leq 3\textbf{E}\norm{\int\limits_{0}^{t}(t-r)^{\alpha-1}\mathscr{E}_{\alpha-\beta,\alpha,\alpha}^{A,B}(t-r)\left( X(r)-Y(r)\right) \mathrm{d}r}^{2}\nonumber\\
&+3\textbf{E}\norm{\int\limits_{0}^{t}(t-r)^{\alpha-1}\mathscr{E}_{\alpha-\beta,\alpha,\alpha}^{A,B}(t-r)\left(b(r,X(r))-b(r,Y(r)) \right) \mathrm{d}r}^{2}\nonumber\\
&+3\textbf{E}\norm{\int\limits_{0}^{t}(t-r)^{\alpha-1}\mathscr{E}_{\alpha-\beta,\alpha,\alpha}^{A,B}(t-r)\left(\sigma(r,X(r))-\sigma(r,Y(r)) \right) \mathrm{d}W(r)}^{2}.
\end{align*}
By the Cauchy-Schwarz inequality, we have 
\begin{align*}
&\textbf{E}\norm{\int\limits_{0}^{t}(t-r)^{\alpha-1}\mathscr{E}_{\alpha-\beta,\alpha,\alpha}^{A,B}(t-r)\left( X(r)-Y(r)\right) \mathrm{d}r}^{2}\leq\mathscr{M}^{2}\int_{0}^{t}(t-r)^{2\alpha-2}\textbf{E}\|X(r)-Y(r)\|^{2}\mathrm{d}r.
\end{align*}
Using Cauchy-Schwarz inequality and Assumption \ref{A1}, we obtain
\begin{align*}
&\textbf{E}\norm{\int\limits_{0}^{t}(t-r)^{\alpha-1}\mathscr{E}_{\alpha-\beta,\alpha,\alpha}^{A,B}(t-r)\left(b(r,X(r))-b(r,Y(r)) \right) \mathrm{d}r}^{2}\leq \mathscr{M}^{2}L_{b}^{2}T\int_{0}^{t}(t-r)^{2\alpha-2}\textbf{E}\|X(r)-Y(r)\|^{2}\mathrm{d}r.
\end{align*}
Moreover, by It\^{o}'s isometry and Assumption \ref{A1}, we also have 
\begin{align*}
&\textbf{E}\norm{\int\limits_{0}^{t}(t-r)^{\alpha-1}\mathscr{E}_{\alpha-\beta,\alpha,\alpha}^{A,B}(t-r)\left(\sigma(r,X(r))-\sigma(r,Y(r)) \right) \mathrm{d}W(r)}^{2}\leq\mathscr{M}^{2} L_{\sigma}^{2} \int_{0}^{t}(t-r)^{2\alpha-2}\textbf{E}\|X(r)-Y(r)\|^{2}\mathrm{d}r.
\end{align*}
Then for all $t \in [0,T]$, we acquire
\begin{align*}
\textbf{E}\|(\mathcal{T}_{\eta}X)(t)-(\mathcal{T}_{\eta}Y)(t)\|^{2}&\leq 3\mathscr{M}^{2}\Big(1+L_{b}^{2}T+L_{\sigma}^{2}\Big)\int_{0}^{t}(t-r)^{2\alpha-2}\textbf{E}\|X(r)-Y(r)\|^{2}\mathrm{d}r,
\end{align*}
which together with the definition of $\|\cdot\|_{\omega}$ as in \eqref{maxnorm} implies that
\begin{align*}
\frac{\textbf{E}\|(\mathcal{T}_{\eta}X)(t)-(\mathcal{T}_{\eta}Y)(t)\|^{2}}{E_{2\alpha-1}(\omega t^{2\alpha-1})}&\leq 3\mathscr{M}^{2}\Big(1+L_{b}^{2}T+L_{\sigma}^{2}\Big)\frac{1}{E_{2\alpha-1}(\omega t^{2\alpha-1})} \\
&\times \int_{0}^{t}(t-r)^{2\alpha-2}E_{2\alpha-1}(\omega r^{2\alpha-1})\frac{\textbf{E}\|X(r)-Y(r)\|^{2}}{E_{2\alpha-1}(\omega r^{2\alpha-1})}\mathrm{d}r.
\end{align*}
By virtue of Lemma \ref{lemtech}, we have for all $t \in [0,T]$:
\begin{align*}
\frac{\textbf{E}\|(\mathcal{T}_{\eta}X)(t)-(\mathcal{T}_{\eta}Y)(t)\|^{2}}{E_{2\alpha-1}(\omega t^{2\alpha-1})}&\leq \frac{3\Gamma(2\alpha-1)}{\omega} \mathscr{M}^{2}\Big(1+L_{b}^{2}T+L_{\sigma}^{2}\Big)\|X-Y\|^{2}_{\omega}
\end{align*}
As a consequence, 
\begin{align*}
\|\mathcal{T}_{\eta}X-\mathcal{T}_{\eta}Y\|^{2}_{\omega} \leq \zeta\|X-Y\|^{2}_{\omega}
\end{align*}
 where 
 \begin{equation*}
  \zeta \coloneqq \frac{3\Gamma(2\alpha-1)}{\omega}\mathscr{M}^{2}\Big(1+L_{b}^{2}T+L_{\sigma}^{2}\Big).
 \end{equation*}
 By \eqref{omega}, we have $\zeta<1$ and the operator $\mathcal{T}_{\eta}$ is a contractive mapping on $H^{2}([0,T],\|\cdot\|_{\omega})$. Using the Banach's fixed point theorem, there exists a unique fixed point of this map in $H^{2}([0,T],\mathbb{R}^{n})$. This fixed point is also a unique solution of \eqref{fstoc} with initial conditions $X(0)=\eta$. The proof of (i) is complete.
 
(ii) Choose and fix $T>0$ and $\eta, \gamma \in \Xi_{0}$. Since $\varphi(\cdot, \eta)$ and $\varphi(\cdot, \gamma)$ are solution of \eqref{fstoc}, it follows that
\begin{align*}
\varphi(t,\eta)-\varphi(t,\gamma)&= (\eta -\gamma)\left( I+t^{\alpha}\mathscr{E}_{\alpha-\beta,\alpha,\alpha+1}^{A,B}(t)B\right)\\
&+\int\limits_{0}^{t}(t-r)^{\alpha-1}\mathscr{E}_{\alpha-\beta,\alpha,\alpha}^{A,B}(t-r)\left(\varphi(r,\eta)-\varphi(r,\gamma)\right) \mathrm{d}r\nonumber\\
&+\int\limits_{0}^{t}(t-r)^{\alpha-1}\mathscr{E}_{\alpha-\beta,\alpha,\alpha}^{A,B}(t-r)\left( b(r,\varphi(r,\eta))-b(r,\varphi(r,\gamma))\right)\mathrm{d}r\nonumber\\
&+\int\limits_{0}^{t}(t-r)^{\alpha-1}\mathscr{E}_{\alpha-\beta,\alpha,\alpha}^{A,B}(t-r)\left( \sigma(r,\varphi(r,\eta))-\sigma(r,\varphi(r,\gamma))\right)\mathrm{d}W(r).
\end{align*}
Hence, using the Jensen's inequality \eqref{ineqq} for $n=5$, Assumption \ref{A1} and \ref{A11}, the Cauchy-Schwarz inequality and It\^{o}'s isometry, we obtain
\allowdisplaybreaks
\begin{align*}
\textbf{E}\|\varphi(t,\eta)-\varphi(t,\gamma)\|^{2}
&\leq  4\mathcal{C}\textbf{E}\|\eta -\gamma\|^{2}\\
&+4\frac{\Gamma(2\alpha-1)}{\omega}\mathscr{M}^{2}\Big(1+L_{b}^{2}T+L_{\sigma}^{2}\Big)\int_{0}^{t}(t-r)^{2\alpha-2}\textbf{E}\|\varphi(r,\eta)-\varphi(r,\gamma)\|^{2} \mathrm{d}r,
\end{align*}
where $\mathcal{C}\coloneqq\left( I+t^{\alpha}\mathscr{E}_{\alpha-\beta,\alpha,\alpha+1}^{A,B}(t)B\right)^{2}$ .

By virtue of Lemma \ref{lemtech} and the definition of $\|\cdot\|_\omega$, we have 
\begin{align*}
\frac{\textbf{E}\|\varphi(t,\eta)-\varphi(t,\gamma)\|^{2}}{E_{2\alpha-1}(\omega t^{2\alpha-1})}&\leq 4\mathcal{C}\|\eta -\gamma\|^{2}_{\text{ms}}+ \zeta  \|\varphi(t,\eta)-\varphi(t,\gamma)\|^{2}_{\omega}.
\end{align*}

Thus, by \eqref{omega}, we have 
\begin{align*}
\Big(1-\zeta \Big)\|\varphi(t,\eta)-\varphi(t,\gamma)\|^{2}_{\omega}& \leq  4\mathcal{C}\|\eta -\gamma\|^{2}_{\text{ms}}.
\end{align*}
Hence,
\begin{align*}
\lim\limits_{\eta \to \gamma}\sup\limits_{t\in [0,T]}\|\varphi(t,\eta)-\varphi(t,\gamma)\|^{2}_{\text{ms}}=0.
\end{align*}
The proof is complete. \enspace $\square$

\section{Asymptotic separation between mild solutions of \eqref{fstoc}} \label{asymptotic}
\begin{thm}
Let $\eta, \gamma \in \Xi_{0}$ such that $\eta \neq \gamma$. Then for any $\epsilon >0$
\begin{equation*}
\limsup_{t\to \infty}t^{\alpha+\epsilon}\|\varphi(t, \eta)-\varphi(t, \gamma)\|_{\text{ms}}=\infty.
\end{equation*}
\end{thm}
\begin{proof}
	Assume the contrary, i.e. there exists a positive constant $\lambda >\alpha$ such that
	\begin{equation*}
	\limsup_{t \to \infty} t^{\lambda}\|\varphi(t, \eta)-\varphi(t, \gamma)\|_{\text{ms}}< \infty,
	\end{equation*}
	for some $\eta, \gamma \in \Xi_{0}$, $\eta\neq \gamma$. There exists constants $T>0$ and $\kappa>0$ such that
	\begin{equation}\label{ineqlambda}
	\|\varphi(t, \eta)-\varphi(t, \gamma)\|^{2}_{\text{ms}}\leq \kappa t^{-2\lambda} \quad \text{for all} \quad t \geq T.
	\end{equation}
	From \eqref{integral equation} and the inequality \eqref{ineqq}, we have 
	\begin{align*}
	\|\eta -\gamma\|^{2}\leq \frac{1}{4\mathcal{C}}\|\varphi(t,\eta)-\varphi(t,\gamma)\|^{2}&+\frac{1}{\mathcal{C}}\norm{\int\limits_{0}^{t}(t-r)^{\alpha-1}\mathscr{E}_{\alpha-\beta,\alpha,\alpha}^{A,B}(t-r)\left( \varphi(r,\eta)-\varphi(r,\gamma)\right) \mathrm{d}r}^{2}\nonumber\\
	&+\frac{1}{\mathcal{C}}\norm{\int\limits_{0}^{t}(t-r)^{\alpha-1}\mathscr{E}_{\alpha-\beta,\alpha,\alpha}^{A,B}(t-r)\left( b(r,\varphi(r,\eta))-b(r,\varphi(r,\gamma))\right)\mathrm{d}r}^{2}\nonumber\\
	&+\frac{1}{\mathcal{C}}\norm{\int\limits_{0}^{t}(t-r)^{\alpha-1}\mathscr{E}_{\alpha-\beta,\alpha,\alpha}^{A,B}(t-r)\left( \sigma(r,\varphi(r,\eta))-\sigma(r,\varphi(r,\gamma))\right)\mathrm{d}W(r)}^{2}.
	\end{align*}

Taking expectation of both sides and using Assumption \ref{A1}, we obtain 
\begin{align*}
\|\eta -\gamma\|^{2}_{\text{ms}}\leq \frac{1}{4\mathcal{C}}\textbf{E}\|\varphi(t,\eta)-\varphi(t,\gamma)\|^{2}&+\frac{1}{\mathcal{C}}\textbf{E}\left( \int\limits_{0}^{t}(t-r)^{\alpha-1}\|\mathscr{E}_{\alpha-\beta,\alpha,\alpha}^{A,B}(t-r)\|\|\varphi(r,\eta)-\varphi(r,\gamma)\|\mathrm{d}r\right) ^{2}\nonumber\\
&+\frac{1}{\mathcal{C}}\textbf{E}\left( \int\limits_{0}^{t}(t-r)^{\alpha-1}\|\mathscr{E}_{\alpha-\beta,\alpha,\alpha}^{A,B}(t-r)\|L_{b}\|\varphi(r,\eta)-\varphi(r,\gamma)\|\mathrm{d}r\right) ^{2}\nonumber\\
&+\frac{1}{\mathcal{C}}\textbf{E}\left( \int\limits_{0}^{t}(t-r)^{\alpha-1}\|\mathscr{E}_{\alpha-\beta,\alpha,\alpha}^{A,B}(t-r)\|L_{\sigma}\|\varphi(r,\eta)-\varphi(r,\gamma)\|\mathrm{d}W(r)\right) ^{2}.
\end{align*}
From \eqref{ineqlambda}, we derive that $\lim_{t\to\infty}\textbf{E}\|\varphi(r,\eta)-\varphi(r,\gamma)\|^{2}=0$. Thus, to derive contradiction it is enough to show that 

\begin{equation}\label{I1}
\lim_{t\to\infty}\mathcal{I}_{1}(t)=0, \qquad \text{where}\quad  \mathcal{I}_{1}(t)\coloneqq \textbf{E}\left( \int\limits_{0}^{t}(t-r)^{\alpha-1}\|\mathscr{E}_{\alpha-\beta,\alpha,\alpha}^{A,B}(t-r)\|\|\varphi(r,\eta)-\varphi(r,\gamma)\|\mathrm{d}r\right) ^{2},
\end{equation}
\begin{equation}\label{I2}
\lim_{t\to\infty}\mathcal{I}_{2}(t)=0, \qquad \text{where}\quad  \mathcal{I}_{2}(t)\coloneqq \textbf{E}\left( \int\limits_{0}^{t}(t-r)^{\alpha-1}\|\mathscr{E}_{\alpha-\beta,\alpha,\alpha}^{A,B}(t-r)\|L_{b}\|\varphi(r,\eta)-\varphi(r,\gamma)\|\mathrm{d}r\right) ^{2}
\end{equation}
and 
\begin{equation}\label{I3}
\lim_{t\to\infty}\mathcal{I}_{3}(t)=0, \qquad \text{where}\quad  \mathcal{I}_{3}(t)\coloneqq \mathscr{M}^{2}L^{2}_{\sigma}\int_{0}^{t}(t-r)^{2\alpha-2}\|\varphi(r,\eta)-\varphi(r,\gamma)\|^{2}_{\text{ms}}\mathrm{d}r.
\end{equation}
To show \eqref{I1} and \eqref{I2}, choose and fix $\delta \in (\frac{\alpha}{\lambda}, 1-\alpha)$. Note that the existence of such a $\delta$ comes from the fact that $\frac{\alpha}{\lambda}< 1-\alpha$ that is equivalent to the assumption $\lambda>\frac{\alpha}{1-\alpha}$. For $t> \max\left\lbrace  T^{1/\delta},1\right\rbrace $, using Cauchy-Schwarz inequality and the inequality \eqref{ineqlambda}, we have  

\begin{align*}
\mathcal{I}_{1}(t)&\leq  2  \textbf{E}\left( \int\limits_{0}^{t^{\delta}}(t-r)^{\alpha-1}\|\mathscr{E}_{\alpha-\beta,\alpha,\alpha}^{A,B}(t-r)\|\|\varphi(r,\eta)-\varphi(r,\gamma)\|\mathrm{d}r\right) ^{2},\\
&+2 \textbf{E}\left( \int\limits_{t^{\delta}}^{t}(t-r)^{\alpha-1}\|\mathscr{E}_{\alpha-\beta,\alpha,\alpha}^{A,B}(t-r)\|\|\varphi(r,\eta)-\varphi(r,\gamma)\|\mathrm{d}r\right) ^{2},\\
&\leq 2\mathscr{M}^{2}\int_{0}^{t^{\delta}}(t-r)^{2\alpha-2}\mathrm{d}r\int_{0}^{t^{\delta}}\|\varphi(r,\eta)-\varphi(r,\gamma)\|^{2}_{\text{ms}}\mathrm{d}r\\
&+2\mathscr{M}^{2}\int_{t^{\delta}}^{t}(t-r)^{2\alpha-2}\mathrm{d}r\int_{t^{\delta}}^{t}\|\varphi(r,\eta)-\varphi(r,\gamma)\|^{2}_{\text{ms}}\mathrm{d}r.
\end{align*}
Since  
\begin{align*}
\int_{0}^{t^{\delta}}(t-r)^{2\alpha-2}\mathrm{d}r=\frac{t^{\delta}}{(t-t^{\delta})^{2-2\alpha}},\quad \int_{t^{\delta}}^{t}(t-r)^{2\alpha-2}\mathrm{d}r=\frac{(t-t^{\delta})^{2\alpha-1}}{2\alpha-1}.
\end{align*}
It yields together with \eqref{ineqlambda} that

\begin{align}\label{ineqI1}
\mathcal{I}_{1}(t)&\leq \frac{2\mathscr{M}^{2}t^{2\delta}}{(t-t^{\delta})^{2-2\alpha}}\sup_{r\in [0,T]}\|\varphi(r,\eta)-\varphi(r,\gamma)\|^{2}_{\text{ms}}+2\kappa \mathscr{M}^{2}\frac{(t-t^{\delta})^{2\alpha-1}}{2\alpha-1}\int_{t^{\delta}}^{t}r^{-2\lambda}\mathrm{d}r\nonumber\\
&\leq \frac{2\mathscr{M}^{2}t^{2\delta}}{(t-t^{\delta})^{2-2\alpha}}\sup_{r\in [0,T]}\|\varphi(r,\eta)-\varphi(r,\gamma)\|^{2}_{\text{ms}}+\frac{2\mathscr{M}^{2}\kappa (t-t^{\delta})^{2\alpha}}{(2\alpha-1)t^{2\delta\lambda}}.
\end{align}

Similarly, we also have
\begin{equation}\label{ineqI2}
\mathcal{I}_{2}(t)\frac{2\mathscr{M}^{2}L^{2}_{b}t^{2\delta}}{(t-t^{\delta})^{2-2\alpha}}\sup_{r\in [0,T]}\|\varphi(r,\eta)-\varphi(r,\gamma)\|^{2}_{\text{ms}}+\frac{2\mathscr{M}^{2}L^{2}_{b}\kappa (t-t^{\delta})^{2\alpha}}{(2\alpha-1)t^{2\delta\lambda}}. 
\end{equation}

By definition of $\delta$, we have $2\delta<2-2\alpha$ and $2\alpha<2\delta \lambda$. Hence, letting $t \to \infty$ in the equalities \eqref{ineqI1} and \eqref{ineqI2} yields $\lim\limits_{t \to \infty}I_{1}(t)=0$ and $\lim\limits_{t \to \infty}I_{2}(t)=0$. Therefore, \eqref{I1} and \eqref{I2} are proved. 


Concerning the assertion \eqref{I3}, let $t \geq  T$ be arbitrary. By \eqref{ineqlambda}, we have
\begin{align*}
 \mathcal{I}_{3}(t)&\leq  \mathscr{M}^{2}L_{\sigma}^{2}\int_{0}^{T}(t-r)^{2\alpha-2}\|\varphi(r,\eta)-\varphi(r,\gamma)\|^{2}_{\text{ms}}\mathrm{d}r+ \kappa\mathscr{M}^{2}L_{\sigma}^{2}\int_{T}^{t}(t-r)^{2\alpha-2}r^{-2\lambda}\mathrm{d}r\\
 &\leq \mathscr{M}^{2}L_{\sigma}^{2}\frac{T}{(t-T)^{2-2\alpha}}\sup_{r\in [0,T]}\|\varphi(r,\eta)-\varphi(r,\gamma)\|^{2}_{\text{ms}}+ \kappa\mathscr{M}^{2}L_{\sigma}^{2}\int_{T}^{t}(t-r)^{2\alpha-2}r^{-2\lambda}\mathrm{d}r.
\end{align*}
Therefore, 
\begin{equation}\label{ineqI3}
\limsup_{t \to \infty}I_{3}(t)\leq \kappa\mathscr{M}^{2}L_{\sigma}^{2}\limsup_{t \to \infty}\int_{T}^{t}(t-r)^{2\alpha-2}r^{-2\lambda}\mathrm{d}r
\end{equation}
Note that for $t\geq 2T$, we have 
\begin{align*}
\int_{T}^{t}(t-r)^{2\alpha-2}r^{-2\lambda}\mathrm{d}r&= \int_{T}^{t/2}(t-r)^{2\alpha-2}r^{-2\lambda}\mathrm{d}r+\int_{t/2}^{t}(t-r)^{2\alpha-2}r^{-2\lambda}\mathrm{d}r\\
&\leq \frac{2^{2-2\alpha}}{t^{2-2\alpha}}\int_{T}^{t/2}r^{-2\lambda}\mathrm{d}r+\left( \frac{t}{2}\right) ^{-2\lambda}\int_{t/2}^{t}(t-r)^{2\alpha-2}\mathrm{d}r\\
&\leq \frac{2^{2-2\alpha}T^{-2\lambda+1}}{(2\lambda-1)t^{2-2\alpha}}+\frac{1}{2\alpha -1}\left( \frac{t}{2}\right) ^{2\alpha-2\lambda-1},
\end{align*}
which together with \eqref{ineqI3} and the fact that $\alpha \in (\frac{1}{2}, 1)$ and $\lambda>\frac{\alpha}{1-\alpha}>\alpha$, implies that $\lim_{t \to \infty}\mathcal{I}_{3}(t)=0$. Thus, \eqref{I3} is proved and therefore the proof is complete.
\end{proof} 

\begin{remark}
		In special case, $\beta =0$ and $A,B=\Theta$ are zero matrices, the system \eqref{fstoc} can be reduced to the following Caputo fractional stochastic differential equation (Caputo FSDE):
		\begin{equation}\label{fsde}
		\begin{cases}
		\left( \prescript{C}{}D_{0^{+}}^{\alpha}X \right) (t)=b(t,X(t))+\sigma(t, X(t))\frac{\mathrm{d}W(t)}{\mathrm{d}t}, \\
		X(0)=\eta,
		\end{cases}
		\end{equation}
		in which asymptotic separation between solutions of above Caputo FSDE has been discussed with the help of Theorem 2 in \cite{doan}. The main point in \cite{doan} is that $\lambda$ is chosen as $\lambda > \frac{2\alpha}{1-\alpha}$. To the best of our knowledge, we study asymptotic separation between mild solutions rather than integral equations. Moreover, unlike Cong et al. \cite{doan}, we consider more general Caputo FSDEs with non-permutable matrices under weaker condition $\lambda>\alpha$ than \cite{doan}. With regards to this condition, the asymptotic separation between the solutions is greater than $t^{-\alpha-\epsilon}$ as $t \to \infty$ for any $\epsilon>0$.
\end{remark}

\section{Example} \label{sec:example}
Now we provide examples to support the theory developed in the previous sections. We consider the nonlinear fractional stochastic equation system with $\alpha =\frac{3}{4}$ and $\beta=\frac{1}{4}$

\begin{align}\label{43}
&\left( \prescript{C}{}D^{3/4}_{0^{+}}X\right) (t)-\begin{pmatrix} 
0.1 & 0.2 \\
0.3 & 0.4 
\end{pmatrix} \left( \prescript{C}{}D_{0^{+}}^{1/4} X \right) (t)-\begin{pmatrix} 
0.4 & 0.1 \\
0.2 & 0.3 
\end{pmatrix} X(t)=  \begin{bmatrix} \sin X_{1} \\  X_{2}+5 \end{bmatrix}
+ \begin{bmatrix} X_{1}+5 \\  \cos X_{2} \end{bmatrix}\frac{\mathrm{d}W(t)}{\mathrm{d}t}, \quad t \in  [0,1], \\ 
&X(0)= \begin{bmatrix} 3 \\ 5 \end{bmatrix},\nonumber
\end{align}

where $X(t)=(X_{1}(t),X_{2}(t))^{T}$ and $W(t)$ is a Wiener process. By comparison with \eqref{fstoc} we have 

\begin{align*}
&A=
\begin{pmatrix} 
0.1 & 0.2 \\
0.3 & 0.4 
\end{pmatrix} ,
B=
\begin{pmatrix} 
0.4 & 0.1 \\
0.2 & 0.3 
\end{pmatrix} , \\
\quad &b(t,X(t)) =  \begin{bmatrix} \sin X_{1} \\  X_{2}+5 \end{bmatrix}, \quad \sigma(t,X(t))= \begin{bmatrix} X_{1}+5 \\  \cos X_{2} \end{bmatrix}.
\end{align*}
It is obvious that $AB\neq BA$. Then, the unique mild solution of \eqref{43} involving non-permutable matrices is given by 

\begin{align*}
X(t)= &\left( I+t^{\frac{3}{4}}\mathscr{E}_{\frac{1}{2},\frac{3}{4},\frac{7}{4}}^{A,B}(t)B\right) X(0)+\int\limits_{0}^{t}(t-r)^{-\frac{1}{4}}\mathscr{E}_{\frac{1}{2},\frac{3}{4},\frac{3}{4}}^{A,B}(t-r)X(r)\mathrm{d}r\nonumber\\
&+\int\limits_{0}^{t}(t-r)^{-\frac{1}{4}}\mathscr{E}_{\frac{1}{2},\frac{3}{4},\frac{3}{4}}^{A,B}(t-r)b(r,X(r))\mathrm{d}r\nonumber\\
&+\int\limits_{0}^{t}(t-r)^{-\frac{1}{4}}\mathscr{E}_{\frac{1}{2},\frac{3}{4},\frac{3}{4}}^{A,B}(t-r)\sigma(r,X(r))\mathrm{d}W(r).
\end{align*}

It is clear that $b(t,X(t)),\sigma (t,X(t))$ satisfy Assumption \ref{A1} and \ref{A11}. 
All the hypotheses of Theorem 4.1  are so verified and hence, has a mild solution on $H^{2}_{\eta}([0, 1],\mathbb{R}^{n})$. Then we verify the asymptotic separation property between two various solutions of Caputo SMTDEs \eqref{fstoc} with more general condition based on $\lambda > \frac{3}{4}$.

\section{Conclusion} \label{concl}
In this paper we have studied Caputo SMTDEs with matrix coefficients that are not permutable. The fractional orders of differentiation $\alpha$ and $\beta$ are also assumed to be $(\frac{1}{2},1)$ and $(0,1)$, respectively. But they are completely independent of each other. 

The core of this paper is to study asymptotic separation of two mild solutions involving non-permutable matrices of fractional stochastic multi-term differential equations and to provide a proof of the existence and uniqueness of \eqref{fstoc} under some natural assumptions on the coefficients. Moreover, we also determine the asymptotic separation between two different solutions of \eqref{fstoc}, where the asymptotic distance $\infty$ is as $t \to \infty$ when $\lambda > \alpha$, implying that $\lambda$ does not depend on $\beta$. 

As a consequence of the main theorems constructing solution functions to Caputo SMTDE systems, by comparing these results with some existing results in the literature proved from the constant coefficients point of view, we were able to find a more general condition based on $\lambda$ which is valid for a class of fractional stochastic differential equations with a single derivative of fractional order. This is a lucky consequence which forms an interesting result in its own right. 

Although the asymptotic separation of two distinct mild solutions have now been constructed, there remain many other interesting open problems to be considered regarding asymptotics of the solution functions which may be studied by methods analogous to those used for computing asymptotics of Mittag-Leffler type functions in the univariate case. 

Other related directions of research may include more deeply the various relevant function spaces may be useful in the qualitative theory of fractional stochastic differential equations
related to these operators, for example well-posedness and regularity theory \cite{doan-kloeden}.

\section*{Acknowledgement}
The author declares that there is no funding information available.

   \end{document}